\documentclass[showpacs,preprintnumbers,amsmath,amssymb,12pt]{article}
\usepackage{graphicx,amsfonts,amssymb,amsthm,amsmath,graphics,epsfig,pstricks,fancyhdr,fancybox}
\usepackage{dcolumn}
\usepackage{bm}
\usepackage[english]{babel}

\newtheorem{theorem}{Theorem}[section]
\newtheorem{proposition}[theorem]{Proposition}

\theoremstyle{definition}

\newcommand{\é}{\'e}

\def\ito{It\={o}}

\newcommand{\LC}{\leb^2_{\bm C}(\re^n,d^n x)}

\newcommand{\sde}{\textsl{SDE}}

\newcommand{\re}{\bm{R}}

\newcommand{\pro}{\mathbb{P}}

\newcommand{\procl}{(Z_t)_{t\geq0}}

\newcommand{\asp}{\mathbb{E}}
\newcommand{\dir}{\mathcal{E}}
\newcommand{\bor}{\mathcal{B}}

\newcommand{\alg}{\mathcal{F}}

\newcommand{\filta}{(\mathcal{F}_t)_{t\geq0}}

\newcommand{\leb}{\mathcal{L}}

\newcommand{\spp}{(\Omega, \alg,\pro)}

\newcommand{\refeq}[1]{~(\ref{#1})}

\newcommand{\id}{\textit{id}}

\newcommand{\LS}{\textit{L-S}}

\begin{document}
\thispagestyle{empty}

\title{\Huge \textbf{Markov processes\\ and\\ generalized Schr\"odinger equations}}
\author{\textsc{Andrea Andrisani}\\
Dipartimento di Matematica, Universit\`a di Bari\\
via E. Orabona 4, 70125 Bari, Italy\\
 \vspace{10pt}
email: \textit{deflema@yahoo.it}\\
{\textsc{Nicola Cufaro Petroni}}\\
Dipartimento di Matematica and \textsl{TIRES}, Universit\`a di Bari\\
\textsl{INFN} Sezione di Bari\\
via E. Orabona 4, 70125 Bari, Italy\\
email: \textit{cufaro@ba.infn.it}}

\date{revised: October 28, 2011}

\maketitle

\begin{abstract}
\noindent Starting from the forward and backward infinitesimal
generators of bilateral, time-homogeneous Markov processes, the
self-adjoint Hamiltonians of the generalized Schr\"odinger equations
are first introduced by means of suitable Doob transformations.
Then, by broadening with the aid of the Dirichlet forms the results
of the Nelson stochastic mechanics, we prove that it is possible to
associate bilateral, and time-homogeneous Markov processes to the
wave functions stationary solutions of our generalized Schr\"odinger
equations. Particular attention is then paid to the special case of
the L\évy-Schr\"odinger (\LS) equations and to their associated
L\évy-type Markov processes, and to a few examples of Cauchy
background noise.
\end{abstract}




\section{Introduction}\label{se:articolo1}

In a few recent papers~\cite{cufaro6} it has been proposed to broaden the scope of the well known
relation between the Wiener process and the Schr\"odinger
equation~\cite{nelson,feynman,bohmvigier,guerra} to other suitable Markov processes. This idea --
already introduced elsewhere, but essentially only for stable processes~\cite{garbacz,laskin} --
led to a \LS\ (L\évy--Schr\"odinger) equation containing additional integral terms which take into
account the possible jumping part of the background noise. This equation has been presented in the
framework of \emph{stochastic mechanics}~\cite{nelson,guerra} as a model for systems more general
than just the usual quantum mechanics: namely as a true \emph{dynamical theory of L\évy processes}
that can find applications in several physical fields~\cite{applications}. However in the previous
papers~\cite{cufaro6} our discussion was essentially heuristic and rather oriented to discuss the
basic ideas, and to show a number of explicit examples of wave packets solutions of these \LS\
equations in the free case, by pointing out the new features as for instance their time dependent
\emph{multi-modality}. In particular the derivation of the \LS\ equation consistently followed a
time-honored~\cite{albeverio} formal procedure consisting in the replacement of $t$ by an
imaginary time variable $it$. While this usually leads to correct results, however, it is apparent
that it can only be a heuristic, handpicked tactics implemented just in order to see where it
leads, and if the results are reasonable: then -- as already claimed in our previous papers -- a
more solid foundation must be found to give substance to these findings. The aim of the present
paper is in fact to pursue this enquiry by giving a rigorous presentation of the relations between
the \LS\ equations and their background Markov processes.

In the original Nelson papers~\cite{nelson} the Schr\"odinger equation of quantum mechanics was
associated to the diffusion processes weak solutions of the \sde\ (Stochastic Differential
Equations)
\begin{equation}
\label{eq:diffuzione}
 dX_t=b(X_t,t)dt+dW_t
\end{equation}
where $W_t$ is a Wiener process. Our aim is then to analyze how this Nelson approach can be
generalized when a wider class of Markov processes is considered instead of the diffusion
processes\refeq{eq:diffuzione}, and what kind of equations are involved, in particular, when
L\'evy processes are considered instead of $W_t$. The L\'evy
processes~\cite{sato,applebaum,protter,cufaro2} can indeed be considered as the most natural
generalization of the Wiener process: they have stationary, independent increments, and they are
stochastically continuous. The Wiener process itself is a L\évy process, but it essentially
differs from the others because it is the unique with \emph{a.s.}\ (almost surely) continuous
paths: the other L\évy processes, indeed, typically show random jumps all along their
trajectories. In the recent years we have witnessed a considerable growth of interest in non
Gaussian stochastic processes -- and in particular into L\'evy processes -- in domains ranging
from statistical mechanics to mathematical finance. In the physical field the research scope is
presently focused mainly on the stable processes and on the corresponding fractional
calculus~\cite{garbacz,laskin,mainardi}, but in the financial domain a vastly more general type of
processes is at present in use~\cite{cont}, while interesting generalizations seem to be at
hand~\cite{tsallis}. Here we suggest that the stochastic mechanics should be considered as a
dynamical theory of the entire gamut of the \emph{infinitely divisible} (not only stable) L\évy
processes with time reversal invariance, and that the horizon of its applications should be
widened even to cases different from the quantum systems.

This approach presents several advantages: on the one hand the use of general infinitely divisible
processes lends the possibility of having realistic, \emph{finite variances}, a situation ruled
out for non Gaussian, stable processes; on the other there are examples of non stable L\'evy
processes which are connected with the simplest form of the quantum, \emph{relativistic
Schr\"odinger equation}: a link with important physical applications that was missing in the
original Nelson model and was recognized only several years later~\cite{ichinose}. This last
remark shows, among others, that the present inquiry is not only justified by the a desire of
formal generalization, but is required by the need to attain physically meaningful cases that
otherwise would not be contemplated in the narrower precinct of the stable laws. Of course it is
well known that the types of general infinitely divisible laws are not closed under convolution:
when this happens the role of the scale parameters becomes relevant since a change in their values
can not be compensated by reciprocal changes in other parameters, and the process no longer is
scale invariant, at variance with the stable processes. This means that, to a certain extent, a
scale change produces different processes, so that for instance we are no longer free to look at
the process at different time scales by presuming to see the same features. Since, however, the
infinitely divisible distributions can have a finite variance, it easy to prove that the L\'evy
processes generated by these infinitely divisible laws will always have a finite variance which
grows linearly with the time: a feature typical of the ordinary (non anomalous) diffusions, while
the stable non Gaussian processes are bound to show typical (anomalous) super- and sub-diffusive
behavior~\cite{cont}.

To give a rigorous justification of the L-S equation, essentially introduced in~\cite{cufaro6} by
means of an analogy, let us first remark that the original Nelson approach for deriving the
Schr\"odinger equation was based on a deep understanding of the \emph{dynamics} of the stochastic
processes, while this reckoned on new definitions of the kinematical quantities -- \emph{forward
and backward mean velocities} and \emph{mean accelerations} -- that anyway safely revert to the
ordinary ones when the processes degenerate in deterministic trajectories. For the time being,
however, our approach will be rather different: we will not resort openly to an underlying
dynamics, but starting instead with the infinitesimal generators $L$ of a semigroup in a Hilbert
space we will explore on the one hand under what conditions we can associate it to a suitable
Markov process $X_t\in\re^n$ with \emph{pdf} (probability density function) $\rho_t$; and on the
other the formal procedures leading from $L$ to a self-adjoint, bounded from below operator $H$ on
$\LC$ and to a wave function $\Psi_t\in\LC$ which turns out to be a solution of the
\emph{generalized Schr\"odinger equation}
\begin{equation}\label{gseq}
 i\partial_t\Psi_t=H\Psi_t
\end{equation}
with $|\Psi_t|^2=\rho_t,\;\forall t\in\re$. While the first task will be accomplished by resorting
to the properties of the Dirichlet forms $\dir$~\cite{fukushima,rock} that can be defined from
$L$, the second result will be obtained by following the path of the Doob
transformations~\cite{albeverio,morato}.

The paper is then organized as follows: while in the Section~\ref{se:articolo3} we will first
recall the less usual features of the Markov processes of our interest, in the
Section~\ref{dirform} we will introduce their associated infinitesimal generators and Dirichlet
forms, and in the subsequent Section~\ref{se:articolo2} we will briefly summarize the essential
notations about the L\évy processes. In the Section~\ref{sm} we will then recall how, by means of
the Doob transformations previously defined in the Section~\ref{se:articolo4}, the stationary
solutions of the usual Schr\"odinger equation actually admit a stochastic representation in terms
of the diffusion processes\refeq{eq:diffuzione}. Finally in the Section~\ref{se:articolo5} we will
focus our attention on our main result about the L\évy-type~\cite{jacob}, Markov processes
associated to the stationary solutions of the \LS\ (\emph{L\évy-Schr\"odinger}) equation
\begin{equation}
\label{eq:intro4}
 i\partial_t\Psi_t=-L_0\Psi_t+V\Psi_t
\end{equation}
which is a particular form of\refeq{gseq}. Here $V$ is an suitable real function, while for an
infinitely derivable function on $\re^n$ with compact support $f\in C_0^{\infty}(\re^n)$, $L_0$
explicitly operates in the following way
\begin{equation*}
 [L_0f](x)=\alpha_{ij}\partial_{ij}^2f(x)+\int_{y\neq0}\left[f(x+y)-f(x)-\bm 1_{B_1}(y)y_i\partial_if(x)\right]\,\ell(dy)
\end{equation*}
where $\alpha_{ij}$ is a symmetric, positive definite matrix, $\bm 1_{B_1}(y)$ is the indicator of
the set $B_1=\{y\in\re^n\,:\,|y|\leq 1\}$, and $\ell(dy)$ is a \emph{L\'evy
measure}~\cite{sato,applebaum}. The name of the equation\refeq{eq:intro4} is due to the fact that
$L_0$ turns out to be the infinitesimal generator of a symmetric L\'evy process, while $V$ plays
the role of a potential, so that\refeq{eq:intro4} closely resembles the usual Schr\"odinger
equation that one obtains when $L_0$ is the infinitesimal generator of a Wiener process. In the
last Section~\ref{cauchy} a few examples of stationary states of Cauchy-Schr\"odinger equations
with their associated L\évy-type processes are explicitly discussed.

\section{Markov Processes}\label{se:articolo3}

A stochastic processes $X_t$ is usually defined for $t\geq0$, but it will be important for us to
consider also processes defined for every $t\in\re$: we will call them \emph{bilateral processes}.
This will allow us to introduce \emph{forward} and \emph{backward} representations that will be
instrumental to define the suitable symmetric operators and the self-adjoint Hamiltonians needed
in our generalized Schr\"odinger equations. It will be useful, moreover, to recall that we will
call a process $X_t$ \emph{stationary} when all its joint distributions (for a Markov process,
those at one and two times are enough) are invariant for a change of the time origin. In this case
the distributions at one time are invariants, and the conditional (transition) distributions
depend only on the time differences. On the other hand we will call it \emph{time-homogeneous}
when just the conditional (transition) distributions are independent from the time origin and
depend only on the time differences. In this case however the process can possibly be non
stationary when the one-time distributions are not constant, and as a consequence the joint
distributions depend on  the changes of the time origin.

Let then $X=(X_t)_{t\in \re}$ be a bilateral, time-homogeneous
Markov process defined on a probability space $\spp$ endowed with
its natural filtration, and taking values on $(\re^n,\bor(\re^n))$.
We will first of all denote respectively by $p_t$ and $\tilde p_t$
its \emph{forward and backward transition functions} defined as
\begin{equation}\label{eq:definizionetransizione}
  p_t(x,B):=\pro\{X_{s+t}\in B\,|\,X_s=x\},\qquad \tilde p_t(x,B):=\pro\{X_{s-t}\in B\,|\,X_s=x\}
\end{equation}
for $s\in\re$, $t\geq0$, $x\in\re^n$ and $B\in\bor(\re^n)$. We will
say that $\mu$ is an invariant measure for $X$ when
\begin{equation*}
 \mu(B)=\int p_t(x,B)\,\mu(dx) =\int \tilde p_t(x,B)\,\mu(dx)\qquad\qquad t>0
\end{equation*}
for $B\in\bor(\re^n)$. Remark that here $\mu$ is not necessarily supposed to be a
\emph{probability} measure. We will indeed keep open the possibility of $X$ being a stationary
process with a general $\sigma$-finite measure as one time marginal, rather than a strictly
probabilistic one. In this case $X$ actually is an \emph{improper process}, namely a process which
is properly defined as a measurable application from an underlying probabilizable space into a
trajectory space and is adapted to a filtration, but which is endowed with a measure which is not
finite. In particular we will find instrumental the use of the Lebesgue measure on
$(\re^n,\bor(\re^n))$ that turns out to be invariant for many of our semigroups and can
consequently be adopted as the overall measure of the process. This is on the other hand not new
if we think to the case of plane waves solutions of the Schr\"odinger equation in quantum
mechanics.

Given an invariant measure $\mu$ it is possible to prove~\cite{andrisani} that, always for
$t\geq0$, we can define on $\leb^2(\re^n,d\mu)$ endowed with the usual scalar product $\langle
f,g\rangle_\mu$ the two semigroups
\begin{eqnarray*}
 [T_tf](x)&:=&\int f(y)\,p_t(x,dy)=\asp\{f(X_{s+t})\,|\,X_s=x\} \cr
 [\tilde T_tf](x)&:=&\int f(y)\,\tilde
 p_t(x,dy)=\asp\{f(X_{s-t})\,|\,X_s=x\}=\asp\{f(X_{s})\,|\,X_{s+t}=x\}
\end{eqnarray*}
respectively called  \emph{forward and backward semigroups}. We also denote by $(L,D(L))$ and
$(\tilde L,D(\tilde L))$, with the specification of their domains of definition, the corresponding
infinitesimal generators. For these semigroups it is possible to prove that
\begin{equation}
\label{eq:adjoint_re}
 T_t^{\dag}=\tilde T_t\qquad\quad t\geq0
\end{equation}
In particular when $T_t$ is self-adjoint so that $T_t^\dag=\tilde T_t=T_t$ the Markov process
$X_t$ is also called \emph{$\mu$-symmetric}, while we will say that the process is simply
\emph{symmetric} when $\pro_{X_t}(B)=\pro_{X_t}(-B)$ for every $B\in\bor(\re^n)$: these two
notions are however strictly related~\cite{applebaum}. We will moreover call the process
\emph{rotationally invariant} if $P_{X_{t}}(B)=P_{Z_{t}}(\mathbb{O}B)$ for every Borel set $B$ and
for every given orthogonal matrix $\mathbb{O}$.

We finally introduce also the \emph{space-time version}~\cite{sharp}
$Y$ of $X$, namely the process
\begin{equation}
 Y_t=(X_t,\tau_t)
\end{equation}
on $(\re^{n+1},\bor(\re^{n+1}))$, with just one more
\emph{degenerate} component: $\tau_t=t\;a.s.$ It is easy to prove
then that the $Y$ forward and backward semigroups and generators --
now denoted by $T_t^Y,\tilde T_t^Y,L^Y$ and $\tilde{L}^Y$ -- are
defined on the space $\leb^2(\re^{n+1},d\mu\, dt)$ and verify
relations similar to\refeq{eq:adjoint_re}, namely
$(T_t^Y)^{\dag}=\tilde T_t^Y$.  This space-time version $Y$ will be
useful for two reasons: first $Y$ is always
time-homogeneous~\cite{andrisani}, even when $X$ it is not; second
the Doob transforms of combinations of their generators (see the
subsequent Section~\ref{se:articolo4}) will give rise exactly to the
space-time operators needed to recover our generalized Schr\"odinger
equations.

\section{Dirichlet forms}\label{dirform}

Up to now we have defined semigroups starting from suitable, given Markov processes, but in this
paper we will be mainly concerned with the reverse question: under which conditions can we define
a Markov process from a given semigroup $(T_t)_{t\geq0}$  on a real Hilbert space $\mathfrak{H}$
with scalar product $\langle\cdot,\cdot\rangle$ and norm $\|\cdot\|\,$? This well known problem
can be faced in several ways, and we will choose to approach it from the standpoint of the
Dirichlet forms. We refer the reader to classical monographs~\cite{fukushima,rock} for an
extensive discussion about this argument. Let $(\dir,D(\dir))$ be a positive definite, bilinear
form on $\mathfrak H$, endowed with the norm on $D(\dir)$
\begin{equation*}
 \|u\|_1^2:=\dir(u,u)+\|u\|^2
\end{equation*}
Some bilinear forms can naturally be associated to linear operators $L$ in the Hilbert space
$\mathfrak{H}$, for instance as $\dir(u,v)=-\langle Lu,v\rangle$, and we look for operators which
are the generators of a semigroup $T_t$, because this could produce the required link between
Markov processes and bilinear forms. In particular it is well known that semigroup generators are
always associated to \emph{coercive, closed} bilinear forms~\cite{rock}
\begin{theorem}\label{th:coercive_semigroup}
Let $(\dir,D(\dir))$ be a coercive, closed form on $\mathfrak{H}$: then there exist a pair of
operators $(L,D(L))$ and $(\tilde L,D(\tilde L))$ on $\mathfrak{H}$, with
\begin{eqnarray*}
D(L)&:=&\{u\in D(\dir)|\ v\rightarrow \dir(u,v)\mbox{ is continuous w.r.t. }\|\cdot\|_1\mbox{ on }D(\dir)\}\\
D(\tilde L)&:=& \{v\in D(\dir)|\ u\rightarrow \dir(u,v)\mbox{ is continuous w.r.t.
}\|\cdot\|_1\mbox{ on } D(\dir)\}
\end{eqnarray*}
and
\begin{eqnarray}
\label{eq:gen_adjoint}
\dir(u,v)=-\langle Lu,v\rangle,\qquad\quad u\in D(L),\; v\in D(\dir)\\
\dir(u,v)=-\langle u,\tilde Lv\rangle,\qquad\quad u\in D(\dir),\;
v\in D(\tilde L)\label{eq:gen_adjoint2}
\end{eqnarray}
which are the infinitesimal generators of two strongly continuous,
contraction semigroups $(T_t)_{t\geq 0}$ and $(\tilde T_t)_{t\geq
0}$ such that
\begin{equation*}
 T_t^{\dag}=\tilde T_t\qquad\quad\forall\;t\geq0
\end{equation*}
\end{theorem}
\noindent According to the previous theorem, coercive, closed forms generate pairs of strongly
continuous, contractive semigroups. In order to be sure, however, that these semigroups are
associated to some Markov process we must be able to implement them by means of explicit
transition functions. The following result~\cite{rock} shows that a way to realize this program is
to deal with \emph{regular Dirichlet} forms, which are particular cases of coercive, closed forms.
The subsequent exposition could well be proposed for more general Hilbert spaces, but to settle
our notation from now on we will rather limit ourselves just to $\mathfrak{H}=\leb^2(\re^n,d\mu)$.
\begin{theorem}\label{th:dirimarkov}
 Take a regular Dirichlet form $(\dir,D(\dir))$ on $\leb^2(\re^n,d\mu)$, with
its associated strongly continuous semigroups $(T_t)_{t\geq0}$ and $(\tilde T_t)_{t\geq0}$: then
there are two time-homogeneous transition functions $p_t$ and $\tilde{p}_t$ on
$(\re^n,\bor(\re^n))$ such that $\mu$-a.s.
\begin{equation*}
[T_tf](x)=\int f(y)p_t(x,dy)\qquad\qquad
 [\tilde T_tf](x)=\int f(y)\tilde{p}_t(x,dy)
\end{equation*}
for every $t\geq 0$ and $f\in \leb^2(\re^n,d\mu)$.
\end{theorem}

\noindent By means of Theorem \ref{th:dirimarkov} we can then associate to a regular Dirichlet
form two Markov processes $(X_{t})_{t\geq0}$ and $(\tilde X_t)_{t\geq0}$ defined on $\re^n$ -- but
for an arbitrary initial distribution -- respectively by $p_t$ and $\tilde p_t$, and enjoying
several useful properties such as right continuity with left limit and strong Markov property (for
details see \cite{fukushima,rock,sharp}). The transition functions $p_t$ and $\tilde p _t$,
however, could in general be \emph{sub-Markovian}, namely we could have $p_t(x,\re^n)\leq 1$ for
some $x\in \re^n$. To avoid this it can be easily proved, by a general property of strongly
continuous semigroups~\cite{engel}, that if the constant function $u_1=1$ belongs to $D(L)$ and
$D(\tilde L)$, then we find $p_t(x,\re^n)=\tilde p_t(x,\re^n)=1$ for every $x\in \re^n$ if and
only if $Lu_1=\tilde Lu_1=0$. In this case, moreover, we also have that $\mu$ is an invariant
measure for both $p_t$ and $\tilde p_t$. Indeed, since $p_t$ and $\tilde p_t$ are \emph{in
duality} with respect to $\mu$, namely
\begin{equation}\label{eq:simmetriatransizioni}
\int\!\!\!\int f(x)g(y)\,p_t(x,dy)\mu(dx)=\int\!\!\!\int f(y)g(x)\,\tilde{p}_t(x,dy)\mu(dx)
\end{equation}
as we easily deduce from $\tilde T_t=T_t^{\dag}$, by taking $f(x)=1$ and $g(y)=\bm 1_B(y)$ for
$B\in \bor(\re^n)$, it is easy to prove that
\begin{equation*}
    \int p_t(x,B)\mu(dx)=\mu(B)
\end{equation*}
namely that $\mu$ is invariant. The same proof can be adapted to $\tilde p_t$.

Finally under special condition it is also possible to associate to $(\dir,D(\dir))$ a
\emph{single} bilateral Markov process $X=(X_t)_{t\in\re}$ obtained by sewing together
$(X_{t})_{t\geq0}$ and $(\tilde X_t)_{t\geq0}$. This occurs in particular when $\mu$ is an
invartiant probability measure. In this case in fact we first define the canonical process $X_t$
with $t\in\re$ and its distribution $\pro_\mu$ as the Kolmogorov extension of
\begin{eqnarray}\label{eq:pro}
    \pro_{\mu}(X_{t_1}\in B_1)&=&\mu(B_1)\nonumber\\
    \pro_{\mu}(X_{t_1}\in B_1,\ldots,X_{t_k}\in B_k) &=& \int_{B_1\times\ldots\times
B_k}\!\!\!p_{t_k-t_{k-1}}(x_{k-1},dx_{k})\ldots \\
          &&\qquad\qquad\qquad\qquad\ldots p_{t_2-t_1}(x_1,dx_2)\mu(dx_1)\nonumber
\end{eqnarray}
for $B_k\in\bor(\re^n)$, $k\in\bm{N}$, and $t_1\leq t_2\leq \ldots\leq t_k\leq\ldots\,$, and then
we show that $(X_t)_{t\in\re}$ is a Markov process with respect to $\pro_{\mu}$, having $p_t$ and
$\tilde p_t$ respectively as its forward and backward transition functions. Actually it is
straightforward to prove that  $(X_t)_{t\in\re}$ is a Markov process and that $p_t$ is the forward
transition function for it. As for $\tilde p_t$ we have instead to check that
\begin{equation}\label{backwtrans}
 \pro_{\mu}(X_{s-t}\in B|X_s=x)=\asp_{\mu}\{\bm 1_{B}(X_{s-t})|X_s=x\}=\tilde p_t(x,B)
\end{equation}
for every Borel set $B$, $s\in \re$ and $t\geq0$. To this effect it will be enough to remark that,
for every $B_1$ and $B_2$, from\refeq{eq:pro} and\refeq{eq:simmetriatransizioni} it is easy to
show that
\begin{equation*}
 \asp_{\mu}\{\bm 1_{B_2}(X_s)\bm 1_{B_1}(X_{s-t})\}
        =\asp_{\mu}\{\bm 1_{B_2}(X_s)\tilde p_t(X_s,B_1)\}
\end{equation*}
which proves\refeq{backwtrans}. It can be seen moreover~\cite{andrisani} that this unifying
procedure can be adopted even when the invariant measure is not a probability. As a consequence,
according to Theorem~\ref{th:dirimarkov}, the semigroup generators $L$ and $\tilde{L}$ on
$\leb^2(\re^n,d\mu)$ derived through~(\ref{eq:gen_adjoint}) and~(\ref{eq:gen_adjoint2}) from a
regular Dirichlet form $(\dir,D(\dir))$ can be considered as the forward and backward generators
of a single, bilateral Markov process $(X_t)_{t\in\re}$ when $Lu_1=\tilde{L}u_1=0$.

Let us conclude this survey with a few remarks about how to check that a bilinear form
$(\dir,D(\dir))$ actually is a regular Dirichlet form. We first recall that, if the space
$C_0^{\infty}(\re^n)$ of the infinitely derivable real functions with compact support is contained
in $D(\dir)$, it is possible to prove~\cite{rock} that for a symmetric Dirichlet form
$(\dir,D(\dir))$ on $\leb^2(\re^n,d\mu)$ the following \emph{Beurling-Deny formula} holds for
every $f,g\in C_0^{\infty}(\re^n)$:
\begin{eqnarray}
 \label{eq:simmform}
\dir(f,g)&=&\sum_{i,j=1}^n\int \partial_{i}f(x)\partial_{j}g(x)\mu^{ij}(dx)+\int f(x)g(x)k(dx)\nonumber\\
 & &\qquad\qquad\qquad+\int\!\!\!\int_{x\neq z}[f(x)-f(z)][g(x)-g(z)]J(dx,dz)
\end{eqnarray}
where $k(dx)$ is a positive Radon measure on $\re^n$ (\emph{killing measure}), $J(dx,dz)$ is a
symmetric, positive Radon measure defined on $\re^n\times\re^n$ for $x\neq z$ (\emph{jump
measure}) and such that for every $f\in C_0^{\infty}(\re^n)$
\begin{equation}\label{eq:verificadiri1}
 \int\!\!\int_{x\neq z}|f(x)-f(z)|^2J(dx,dz)<\infty
\end{equation}
and $\mu^{ij}(dx)$ for $i,j=1,\ldots,n$ are positive Radon measures
on $\re^n$ (\emph{diffusion measures}) such that for every compact
subset $B\subseteq\re^n$ and $b\in\re^n$ it results
\begin{equation}\label{eq:verificadiri2}
 \mu^{ij}(B)=\mu^{ji}(B)\qquad\qquad\sum_{i,j=1}^{n}b_ib_j\mu^{ij}(B)\geq 0
\end{equation}
For us however it is important to recall that, if our form is also a \emph{closable}
one~\cite{rock}, it is possible to prove a sort of reverse statement so that we will practically
be able use\refeq{eq:simmform} to check that $\dir$ is a regular Dirichlet form.
\begin{proposition}\label{pr:closable-diri}
 Take a closable, bilinear form
$(\dir,D(\dir))$ on $\leb^2(\re^n,d\mu)$ with $D(\dir)=
C_0^{\infty}(\re^n)$: if $\dir$ satisfies the Beurling-Deny
formula\refeq{eq:simmform}, then its closure
$(\bar{\dir},D(\bar{\dir}))$ is a regular Dirichlet form.
\end{proposition}
\noindent Remark moreover that closability can often be verified in a very simple way, as shown by
the following result~\cite{rock}
\begin{proposition} \label{pr:chiudibile} Let $(L,D(L))$ be a symmetric, negative
definite linear operator on $\mathfrak{H}$, and define the bilinear form $\dir(u,v):=-\langle
Lu,v\rangle$ with $D(\dir):=D(L)$. Then $(\dir,D(\dir))$ is closable.
\end{proposition}
\noindent Taken together the Theorem~\ref{th:dirimarkov} and the
Proposition~\ref{pr:closable-diri} imply that when we want to know if there is a Markov process
associated to a given generator $L$, essentially we must first consider the bilinear
forms\refeq{eq:gen_adjoint} and\refeq{eq:gen_adjoint2}, and then check that the Beurling-Deny
formula\refeq{eq:simmform} holds.

\section{L\'evy Processes}\label{se:articolo2}

The processes with independent increments constitute an important class of Markov processes, and
among them the L\évy processes~\cite{sato,applebaum,protter,cufaro2} are of particular relevance
and are today widely applied in a variety of fields~\cite{applications,cont}; they have also been
introduced in the larger context of quantum probability~\cite{accardi}. Actually these processes
are also time-homogeneous and hence they can be identified by means of a transition function $p_t$
only. They are stochastically continuous with independent and stationary increments, and they
include many families of well known processes as the Poisson, the Cauchy, the Student, the
Variance-Gamma and, of course, the Wiener process which is the unique Gaussian process, with
$a.s.$ continuous paths. The L\'evy processes are Feller processes and -- if integrable with zero
expectation -- they are also martingales. In general, however, they always are semi-martingales,
so that in any case they can be used as integrators in \ito\ integrals~\cite{protter}. It is
finally important to remark that between the L\'evy processes and the infinite divisible
distributions~\cite{sato} there is a one-to-one correspondence.

Given a filtered, complete probability space $(\Omega,\alg,\filta,\pro)$, let us consider a L\évy
process $(Z_{t})_{t\geq0}$ with L\évy measure $\ell$.
It is well known then that the following \emph{L\'evy-Khintchin formula}~\cite{sato} holds
\begin{theorem}\label{th:levi_khintchine}
Let $(Z_t)_{t\geq0}$ a L\'evy process. Then its logarithmic
characteristic $\eta(u):=\log\asp\{e^{iu\cdot Z_1}\}$ satisfies
\begin{equation} \label{eq:levi_khintchine}
 \eta(u)=-\frac{1}{2}u\cdot Au+i\gamma\cdot u+\int_{y\neq0}\left[e^{iu\cdot
y}-1-i(u\cdot y)\,\bm 1_{B_1}(y)\right]\ell(dy)
\end{equation}
where $A=\|\alpha_{ij}\|$ is a symmetric non negative-definite $n\times n$ matrix,
$\gamma\in\re^n$, $B_1=\{y\in\re^n\,:\,|y|\leq1\}$ and $\ell(dy)$ is a L\'evy measure. This
representation of $\eta$ by means of the triplet $(A,\ell,\gamma)$ is unique. Viceversa, if $A$ is
a symmetric non negative-definite $n\times n$ matrix, $\gamma\in\re^n$ and $\ell$ is a L\'evy
measure, then (\ref{eq:levi_khintchine}) is the logarithmic characteristic associated to a unique
(in distribution) L\'evy process $(Z_t)_{t\geq0}$.
\end{theorem}
\noindent The triplet $(A,\ell,\gamma)$ is usually called the \emph{generating triplet} of
$\procl$. The laws $\pro_{Z_1}$ with characteristic functions
$\varphi(u)=e^{\eta(u)}=\asp\{e^{iu\cdot Z_1}\}$ turn out to be \emph{infinitely divisible
(\id)}~\cite{sato}, and as a consequence $\varphi^t$ still is the characteristic function of some
other \id\ law for every $t>0$. More precisely, if $\varphi_t$ is the characteristic function of
$\pro_{Z_t}$ for $t\geq0$, namely $\varphi_t(u)=\asp\{e^{iu\cdot Z_t}\}$, then it is easy to show
that $\varphi_t=\varphi^t$. This result provides a one-to-one correspondence between the L\évy
processes and the \id\ laws~\cite{sato,cufaro2} in such a way that every L\'evy process is in fact
uniquely determined by its \id\ distribution at a unique instant, usually at $t=1$.

We conclude this section with the explicit expressions of the infinitesimal generator
$(L_0,D(L_0))$ and of the Dirichlet form $(\dir_0,D(\dir_0))$ associated to a symmetric L\'evy
process $Z_t$ in $\leb^{2}(\re^{n},dx)$. The generator $L_0$ is a pseudo-differential operator
with symbol $\eta$~\cite{applebaum,protter} namely
\begin{equation}
 [L_0f](x)=\frac{-1}{(2\pi)^n}\int
e^{iux}\hat{f}(u)\eta(u)du \label{eq:levval}
\end{equation}
where  $\hat{f}$ is the Fourier transform of $f$, and $D(L_0)$ is
the set of the $f\in \leb^{2}(\re^{n},dx)$ such that
\begin{equation}\label{eq:levdom}
    \int|\hat{f}(u)|^{2}|\eta(u)|^{2}du<\infty
\end{equation}
From\refeq{eq:levdom} it can be proved that actually the Schwarz space $\mathcal{S}(\re^{n})$ is a
subset of $D(L_0)$, and when $f\in\mathcal{S}(\re^{n})$ the infinitesimal generator of $Z_t$ takes
the more explicit form
\begin{equation}\label{eq:levval2}
 [L_0f](x)=\frac{1}{2}\nabla\cdot A\nabla f(x)+\int_{y\neq0}[\delta_y^2f](x)\,\ell(dy)
\end{equation}
where $A=\|\alpha_{ij}\|$ is the symmetric, positive definite matrix, and $\ell$ is the Levy
measure of the generating triplet of $Z_t$, and we adopted the shorthand notations
\begin{eqnarray*}
 [\delta_yf](x)&:=&f(x+y)-f(x)\cr
 [\delta_y^2f](x)&:=&f(x+y)-f(x)-y\cdot\nabla f(x)\bm1_{B_1}(y)
\end{eqnarray*}
with $B_1=\{y\in\re^n:\,|y|\leq1\}$. It is also easy to see then
that
\begin{eqnarray}
  \delta_y (fg)&=&g\delta_y f+f\delta_y g+\delta_yf\delta_y g\label{eq:propr_delta}\\
 \delta_y^2(fg)&=& g
 \delta_y^2f+f\delta_y^2g+\delta_yf\delta_yg\label{eq:provalemma0}
\end{eqnarray}
If the L\évy process is also rotationally invariant then we have $\alpha_{ij}=\alpha\delta_{ij}$
and\refeq{eq:levval2} is reduced to
\begin{equation}\label{eq:levval3}
 [L_0f](x)=\frac{\alpha}{2}\,\nabla^2 f(x)+\int_{y\neq0}[\delta_y^2f](x)\,\ell(dy)
\end{equation}
As for the bilinear form associated to our symmetric L\'evy process
$Z_t$ with generator $(L_0,D(L_0))$, namely
\begin{equation*}
    \dir_0(f,g)=-\langle L_0f,g\rangle=-\int g(x)[L_0f](x)\,dx
\end{equation*}
we have \cite{fukushima} on $D(L_0)$
\begin{equation} \label{eq:levdirval}
 \dir_0(f,g) = -\int \nabla g(x)\cdot A\nabla f(x) \,dx-\frac{1}{2}\!\int\!\!
\int_{y\neq0}[\delta_yf](x)[\delta_yg](x)\,\ell(dy)dx
\end{equation}
This last expression can also be extended to a $D(\dir_0)\supseteq
D(L_0)$, the set of the $f\in \leb^{2}(\re^{n},dx)$ such that
\begin{equation}\label{eq:levdirdom}
 \int \nabla f(x)\cdot A\nabla f(x) \,dx+\int\!\!\int_{y\neq0}\big|[\delta_yf](x)\big|^{2}\,\ell(dy)dx<\infty
\end{equation}

\section{Doob transformations}\label{se:articolo4}

We  turn now to the discussion of the association of a generalized Schr\"odinger equations to our
Markov processes. Let $(X_{t})_{t\in\re}$ be a time-homogeneous, bilateral Markov process with
infinitesimal forward and backward generators $(L,D(L))$ and $(\tilde L,D(\tilde L))$. We suppose
that
\begin{enumerate}
  \item\label{hp1} $X_{t}$ has an $a.c.$\ invariant measure
  $\mu(dx)=\rho(x)dx$;
    \item $\rho(x)>0$ $\;a.s.$\ with respect to the Lebesgue measure, so that $\mu(dx)$ is equivalent to the Lebesgue
    measure;
    \item\label{hp3} the set $D(L)\cap D(\tilde L)$ is dense in
    $\leb^{2}(\re^{n},d\mu)$.
\end{enumerate}
As already stated in the Section~\ref{se:articolo3} the invariant
measure $\mu$ is not necessarily required to be a probability
measure: it will be made clear in a few subsequent examples about
the plane waves that we will indeed consider also cases where the
invariant measure is rather $\sigma$-finite, as the Lebesgue measure
on $\re^n$. From our hypothesis~\ref{hp3} it also
follows~\cite{engel} that both $(T_{t})_{t\geq0}$ and
$(\tilde{T}_{t})_{t\geq0}$ are \emph{strongly continuos} semigroups
in $\leb^{2}(\re^{n},d\mu)$, and that
\begin{equation}\label{eq:aggiungi2}
 \tilde L=L^{\dag}
\end{equation}
If then $Y_t$ is the space-time version of $X_{t}$, its
infinitesimal forward and backward generators $L^Y$ and
$\tilde{L}^Y$ are defined on $\leb^{2}(\re^{n+1},d\mu\, dt)$, and it
can be easily shown that
\begin{equation}\label{eq:aggiungi3}
    L^Y=L+\partial_t\qquad\quad \tilde{L}^Y=\tilde{L}-\partial_t
\end{equation}
with $D(L^Y)= D(L)\otimes \mathcal{H}^{1}$ and $D(\tilde{L}^Y)=
D(\tilde{L})\otimes \mathcal{H}^{1}$, where $\mathcal{H}^{1}$ is the
space of the absolutely continuous functions of $t$ with a square
integrable derivative. The result\refeq{eq:aggiungi3} is apparently
connected to the fact that the infinitesimal generator of the
translation semigroup -- namely the semigroup of the degenerate
Markov process $\tau_{t}$ -- is the time derivative~\cite{engel}.
Remark that if the hypothesis~\ref{hp3}.\ is satisfied by $L$ and
$\tilde L$, then it will hold also for $L^Y$ and $\tilde L^Y$, and
we will have
\begin{equation}
\label{eq:aggio} \tilde L^Y=(L^Y)^{\dag}
\end{equation}
All our generators $L,\tilde{L},L^Y$ and $\tilde L^Y$ can now be
also naturally extended to the respective spaces
$\leb_{\bm{C}}^{2}(\re^{n},d\mu)$ and
$\leb_{\bm{C}}^{2}(\re^{n+1},d\mu\,dt)$ of the complex valued
functions, where they are still densely defined and closed, while
the relations\refeq{eq:aggiungi2} and\refeq{eq:aggio} are preserved.
Moreover, in order to make sure that the Hamiltonian operators that
we will introduce later are bounded from below, we also add a fourth
hypothesis:
\begin{enumerate}
  \item[4.]\label{hp4} $\exists\, C\in\re$ such that for every $\phi$ in $D(L)$
  \begin{equation}
  \label{eq:55bis}
   -\Re\langle\phi,L\phi\rangle_\mu+\Im\langle\phi,L\phi\rangle_\mu\geq C\left\|\phi\right\|_\mu
   \end{equation}
\end{enumerate}
Remark however that if $X_{t}$ is a $\mu$-symmetric process, namely if $L$ turns out to be
self-adjoint in $\leb^2_{\bm{C}}(\re^n,d\mu)$, then $\Im\langle\phi,L\phi\rangle_\mu$ vanishes and
the condition (\ref{eq:55bis}) becomes
\begin{equation*}
-\langle\phi,L\phi\rangle_\mu\geq C\left\|\phi\right\|_\mu
\end{equation*}
which is automatically satisfied because $L$ is the generator of a
contraction semigroup so that~\cite{engel}
$\langle\phi,L\phi\rangle_\mu\leq0$, and it will be enough to take
$C=0$. As a consequence the hypothesis 4 is always satisfied by
$\mu$-symmetric processes. The relevance of this fourth condition
will be made clear in the following, and it lies mainly in the fact
that for a given symmetric, bounded from below operator
$(H_0,D(H_0))$ on a Hilbert space $\mathfrak{H}$, there always
exists~\cite{reed} a \emph{self-adjoint} operator $(H,D(H))$ bounded
from below such that $D(H_0)\subset D(H)$, and that for every $v\in
D(H_0)$, $Hv=H_0v$. This operator $H$ is usually called the
\emph{Friedrichs extension} of $H_0$.
\begin{proposition}
If $X_{t}$ is a time-homogeneous Markov process satisfying the hypotheses~1-4, then the operator
in $\leb_{\bm{C}}^{2}(\re^{n+1},d\mu\,dt)$
\begin{equation}
\label{eq:6} K^Y:=-\frac{\tilde L^Y+iL^Y}{i+1}=-\frac{L^Y+\tilde L^Y}{2}+\frac{L^Y-\tilde L^Y}{2i}
\end{equation}
defined on $D(K^Y):=D(L^Y)\cap D(\tilde L^Y)$, is symmetric.
\end{proposition}
\begin{proof}
From the hypothesis~\ref{hp3}.\ we deduce that $K^Y$ is densely
defined, while the symmetry easily follows from\refeq{eq:aggio}.
\end{proof}\\
We can also introduce in $\leb_{\bm{C}}^{2}(\re^{n},d\mu)$ the reduced operator
\begin{equation}\label{k}
 K:=-\frac{\tilde L+iL}{i+1}=-\frac{L+\tilde L}{2}+\frac{L-\tilde L}{2i}
\end{equation}
defined on $D(K):=D(L)\cap D(\tilde L)$, and prove in a similar way
that it is symmetric. Remark how the existence of both a forward and
a backward generator is instrumental here in the definition of our
symmetric operators $K$ and $K^Y$. It should also be said, however,
that starting from our generators we could define several different
symmetric operators, our present choice being dictated mainly by an
analogy with the Gaussian case, and by our interest in preserving
the presence of the time derivatives $\partial_t$ in the final
operators.

Given now an arbitrary real function $S(x)$ (a change in it would
simply be tantamount to a gauge transformation) and a constant
$E\in\re$, we first define the \emph{wave functions}
\begin{equation}\label{wavefunct}
\psi(x):=\sqrt{\rho(x)}\,e^{iS(x)}\quad\qquad\Psi(x,t):=\psi(x)\,e^{-iEt}
\end{equation}
and then we remark that the operators
\begin{eqnarray}\label{eq:tipo-doob}
U_{\Psi}&:&f\in \leb^{2}_{\bm{C}}(\re^{n+1},d\mu dt)\longrightarrow f\Psi\in
\leb^{2}_{\bm{C}}(\re^{n+1},dxdt)\\
U_{\psi}&:&f\in \leb^{2}_{\bm{C}}(\re^{n},d\mu)\longrightarrow
f\psi\in \leb^{2}_{\bm{C}}(\re^{n},dx)\nonumber
\end{eqnarray}
are unitary with
\begin{equation*}
U_{\Psi}^\dag=U_{\Psi}^{-1}=U_{\frac{1}{\Psi}}\qquad\qquad
U_{\psi}^\dag=U_{\psi}^{-1}=U_{\frac{1}{\psi}}
\end{equation*}
By means of these we can now introduce the operators
\begin{equation}\label{eq:aggiungi5}
 K_\Psi^Y=U_{\Psi}K^YU_{\Psi}^{-1}\qquad\qquad K_\psi=U_{\psi}KU_{\psi}^{-1}
\end{equation}
acting respectively on $\leb^{2}_{\bm{C}}(\re^{n+1},dxdt)$ and
$\leb^{2}_{\bm{C}}(\re^{n},dx)$ with $D(K^Y_\Psi)=U_{\Psi}D(K^Y)$
and $D(K_\psi)=U_{\psi}D(K)$. These unitary transformations are
reminiscent of the well known \emph{Doob transformation}
\cite{sharp,doob,yor} which is applied to the infinitesimal
generators $L$ of Markov processes for a real, positive $\Psi$ in
the domain of $L$ with $L\Psi=0$. Our transformation could also be
defined in a more general way~\cite{sharp}, but in
fact\refeq{eq:aggiungi5} turns out to be well suited to our purposes
so that we will continue to call it Doob transformation, while
$K^Y_\Psi$ and $K_\psi$ will be called the Doob transforms of $K^Y$
and $K$.

\begin{proposition}\label{prop}
The operator $K^Y_\Psi$ can be written as
\begin{equation*}
K^Y_\Psi= H_0\otimes I_t-iI_0\otimes \partial_t
\end{equation*}
where $I_0$ and $I_t$ are the identity operators respectively on the $x$ and $t$ variables, while
\begin{equation}\label{eq:hamiltoniano}
 H_0=K_\psi+E
\end{equation}
turns out to be a symmetric and bounded from below operator in $\leb^{2}_{\bm{C}}(\re^{n},dx)$. We
also have that $\psi$ is an eigenvector with eigenvalue $E$ of the Friedrichs extension $H$ of
$H_0$
\begin{equation}\label{statschr}
    H\psi(x)=E\psi(x)
\end{equation}
and that the function $\Psi(x,t)$ is a strong solution of
\begin{equation}\label{levyschr}
  i\partial_t\Psi(x,t)=H\Psi(x,t)
\end{equation}
being for every $t>0$ also a solution of\refeq{statschr}.
\end{proposition}
\begin{proof}
From (\ref{eq:aggiungi3}) we have for $\phi\in D(K^Y_\Psi)$
\begin{eqnarray*}
K^Y_\Psi\phi&=&-\Psi\frac{L+\tilde{L}}{2}\frac{\phi}{\Psi}-i\Psi\partial_t\frac{\phi}{\Psi}+\Psi\frac{L-\tilde{L}}{2i}\frac{\phi}{\Psi}\\
 &= &-\psi\left(\frac{L+\tilde{L}}{2}-\frac{L-\tilde{L}}{2i}\right)\frac{\phi}{\psi}+E\phi-i\partial_t\phi\;=\;(H_0\otimes I_t)\phi-i\left(I_0\otimes\partial_t\right)\phi
\end{eqnarray*}
From the hypothesis~4.\ we can see now that $H_0$ is bounded from below, while from the
hypothesis~\ref{hp3}.\ we deduce that it is densely defined, and from (\ref{eq:aggiungi2}) that it
is symmetric: then its Friedrichs extension $H$ exists and is self-adjoint. Moreover the constant
function $\phi_1=1$ is an element of $D(L)\cap D(\tilde L)$ and
from\refeq{eq:simmetriatransizioni} and hypothesis~\ref{hp1} it is easy to see that
$L\phi_1=\tilde L\phi_1=0$. As a consequence we have from~(\ref{eq:hamiltoniano}) that $\psi\in
D(H_0)$, and that $H_0\psi=E\psi$: namely $\psi$ is an eigenfunction of $H_0$ corresponding to the
eigenvalue $E$. It is then straightforward to prove (\ref{levyschr}).
\end{proof}\\
The Friedrichs extension $H$ of $H_0$ will be called in the following the \emph{Hamiltonian
operator} associated to the Markov process $X_{t}$, and the equation~(\ref{levyschr}) will take
the name of \emph{generalized Schr\"odinger equation}. Remark that if in particular $X_{t}$ is a
$\mu$-symmetric process, namely if $L$ is self-adjoint, then we simply have
\begin{equation}\label{eq:hamiltonianosimmetrico}
K^Y=-L-i\partial_t\qquad\quad K=-L\qquad\quad H_0=-U_\psi LU_\psi^{-1}+E
\end{equation}
so that $H_0$ itself is self-adjoint and hence coincides with $H$. This however is not the case
for every Markov process $X(t)$ that we can consider within our initial hypotheses.

In conclusion we have shown that from every Markov process $X_t$ obeying our four initial
hypotheses we can always derive a self-adjoint Hamiltonian $H$ and a corresponding generalized
Schr\"odinger equation\refeq{levyschr}. In this scheme the process $X_t$ is associated to a
particular stationary solution $\Psi$ of\refeq{levyschr} in such a way that $|\Psi|^2$ coincides
with the invariant measure of $X_t$. Vice versa it would be interesting to be able to trace back a
suitable Markov process $X_t$ from a solution $\Psi$ -- at least from a \emph{stationary} solution
-- of\refeq{levyschr} with a self-adjoint Hamiltonian. In fact, even when from a given Hamiltonian
$H$ and a stationary solution $\Psi$ of\refeq{levyschr} we can manage -- by treading back along
the path mapped in this section -- to get a semigroup $L$, we are still left with the problem of
checking that the minimal conditions are met in order to be sure that there is a Markov process
$X_t$ associated to $L$. In this endeavor the previous discussion about the Dirichlet forms
developed in the Section~\ref{dirform} will turn out to be instrumental as it will be made clear
in the following.

\section{Stochastic mechanics}\label{sm}

Let us begin by remembering that the names we gave to $H$ and to the equation\refeq{levyschr} at
the end of the previous section are justified by the fact that when $X_{t}$ is a solution of the
\ito\ \emph{SDE} (Stochastic Differential Equation)
\begin{equation}\label{ito}
dX_{t}=b(X_{t})dt+dW_{t}
\end{equation}
where $W_t$ is a Wiener process, then $H$ turns out to be the Hamiltonian operator appearing in
the usual Schr\"odinger equation of quantum mechanics. To show how this works we will briefly
review here this well known result in the less familiar framework of the Doob
transformations~\cite{albeverio,morato} because -- at variance with the original Nelson
\emph{stochastic mechanics} -- this procedure allows a derivation of the Schr\"odinger equation
without introducing an explicit \emph{dynamics} that, at present, is still not completely ironed
out in the more general context of the L\évy processes.

Let us start by considering the operator
\begin{equation*}
  L_{0}:=\frac{1}{2}\,\nabla^2
\end{equation*}
with $D(L_0)$ the set of all the functions $f$ that, along with their first and second generalized
derivatives, belong $\leb^{2}(\re^{n},dx)$. It is well known that $(L_0,D(L_0))$ is the
infinitesimal generator of a Wiener process. If now we take $\phi\in D(L_{0})$ such that
\begin{equation}
 \label{eq:condifi1}
\int_{\re^{n}}\phi^2(x)dx=1,\qquad\quad \phi\neq0\quad \mbox{$a.s.$\ in}\ dx
\end{equation}
we can define in $\leb^{2}(\re^{n},d\mu)$ with $\mu(dx)=\rho(x)dx=\phi^{2}(x)dx$ a second operator
\begin{equation}\label{eq:diffoval}
 Lf:=\frac{L_{0}(\phi f)-fL_0\phi}{\phi}
\end{equation}
where $D(L):=C^{\infty}_{0}(\re^{n})$ is the set of infinitely derivable functions on $\re^{n}$
with bounded support.
It is straightforward now to see that $L$ is correctly defined~\cite{andrisani}, and
that\refeq{eq:diffoval} can be recast in the form
\begin{equation}\label{eq:diffovalbis}
 Lf=\frac{\nabla \phi}{\phi}\cdot\nabla f+\frac{1}{2}\nabla^2 f=b\cdot\nabla f+\frac{1}{2}\nabla^2
 f\qquad\qquad b=\frac{\nabla \phi}{\phi}
\end{equation}
which on the other hand is typical for the generators of a process satisfying the
\emph{SDE}\refeq{ito}.  At first sight this seems to imply directly that to every given $\phi$ we
can always associate a Markov process (weak) solution of the \emph{SDE}\refeq{ito} with $b$
defined as in\refeq{eq:diffovalbis}, but this could actually be deceptive because this association
is indeed contingent on the properties of the function $b$, and hence of $\phi$. To be more
precise: if we know that the equation\refeq{ito} has a solution $X_t$, then its generator
certainly has the form\refeq{eq:diffovalbis}; but vice versa if an operator\refeq{eq:diffovalbis}
-- or a \emph{SDE}\refeq{ito} -- is given with an arbitrary $b$, we can not in general be sure
that a corresponding Markov process $X_t$ solution of\refeq{ito} does in fact exist, albeit in a
weak sense. In the light of these remarks it is important then to be able to prove, by means of
the Dirichlet forms, that for $L$ and $b$ defined as in\refeq{eq:diffovalbis} we can always find a
Markov process solution of\refeq{ito}.
\begin{theorem}\label{th:esempio1}
The operator $(L,D(L))$ defined in (\ref{eq:diffoval}) is closable and its closure $(\bar
L,D(\bar L))$ is a self-adjoint, negative definite operator which is the infinitesimal generator
of a Markov process $(X_t)_{t\in\re}$ (weak) solution of\refeq{ito}. The measure $\mu(dx)$ is
invariant for this Markov process.
\end{theorem}
\begin{proof}
Let us consider on $\leb^{2}(\re^{n},d\mu)$ the bilinear form
\begin{eqnarray*}
\mathcal{E}(f,g):=-\langle Lf,g\rangle_\mu
\end{eqnarray*}
with $D(\dir):=C^{\infty}_{0}(\re^{n})$. For $f,g \in D(\dir)$ with
an integration by parts we get
\begin{eqnarray*}
\mathcal{E}(f,g)&=&-\frac{1}{2}\int_{\re^{n}} g(x)\phi(x)\nabla^2 (\phi f)(x)dx+\frac{1}{2}\int_{\re^{n}}g(x)f(x)\phi(x)\nabla^2\phi(x)dx\\
&=&-\frac{1}{2}\int_{\re^{n}} g(x)\phi^{2}(x)\nabla^2 f(x)dx-\int_{\re^{n}}g(x)\phi(x) \nabla f(x)\cdot\nabla\phi(x)dx\\
&=&\int_{\re^{n}} \phi^{2}(x)\nabla g(x)\cdot\nabla f(x)dx
\end{eqnarray*}
namely our form satisfies the Beurling-Deny formula\refeq{eq:simmform} with vanishing jump and
killing measures, and with the condition\refeq{eq:verificadiri2} trivially satisfied. As a
consequence $\dir$ is symmetric and positive definite, so that also $-L$ is symmetric and positive
definite and then is also closable~\cite{reed}. Hence by Proposition \ref{pr:chiudibile}
$(\dir,D(\dir))$ is closable and from the Proposition~\ref{pr:closable-diri} its closure
$(\bar{\mathcal{E}},D(\bar{\mathcal{E}}))$ is a symmetric, regular Dirichlet form which from
Theorem~\ref{th:coercive_semigroup} is associated to a self-adjoint, infinitesimal generator
$(\bar{L},D(\bar{L}))$. Of course $(\bar{L},D(\bar{L}))$ itself turns out~\cite{reed} to be the
closure of $(L,D(L))$.

This $\bar{L}$ generates now a unique, bilateral Markov process
$(X_{t})_{t\in\re}$ on $(\re^{n},\mathcal{B}(\re^{n}))$ having $\mu$
as its invariant measure when the conditions discussed in the
remarks following the Theorem~\ref{th:dirimarkov} are met. In this
context we are then left just with the task of showing that the
constant function $f_1(x)=1$ belongs to $D(\bar L)$ and that $\bar L
f_1=0$. In fact for every $f\in C_0^{\infty}$ dense in
$\leb^{2}(\re^{n},d\mu)$, with an integration by parts we have
\begin{eqnarray*}
\langle f_1,\bar{L}f\rangle_\mu&=&\int_{\re^n} [\bar L f ]\phi^2dx=\int_{\re^n} [L f] \phi^2dx\\
 &=&\frac{1}{2}\int_{\re^n}\phi \nabla^2 (f\phi) dx-\frac{1}{2}\int_{\re^n}f\phi\nabla^2 \phi dx\\
 &=&\frac{1}{2}\int_{\re^n} f\phi\nabla^2 \phi dx-\frac{1}{2}\int_{\re^n}f\phi\nabla^2 \phi dx=0
\end{eqnarray*}
Being $\bar L$ self-adjoint this implies first that $f_1\in D(\bar
L)$, and that $\langle \bar{L}f_1,f\rangle_\mu=0$ for every $f\in
C_0^{\infty}$ dense in $\leb^{2}(\re^{n},d\mu)$, and then that $\bar
Lf_1=0$.
\end{proof}\\
Since the process $(X_t)_{t\in\re}$ obtained from $L_0$ in the Theorem~\ref{th:esempio1} satisfies
all the hypotheses~{1-4} with $\rho(x)=\phi^2(x)$, we can now go on looking for the Hamiltonians
produced by the Doob transformations. Take first the wave functions\refeq{wavefunct} with
arbitrary $E\in\re$ and $S(x)=0$, namely
\begin{equation}\label{statsol}
  \psi(x)=\phi(x)\quad\qquad\Psi(x,t)=e^{-iEt}\phi(x)
\end{equation}
Since $\bar{L}$ is self-adjoint, by
applying\refeq{eq:hamiltonianosimmetrico} we easily find as
Hamiltonian associated to $X_{t}$ by the Doob transformation the
operator
\begin{equation}\label{hamiltonian}
 H:=-\frac{1}{2}\nabla^2+V(x)
\end{equation}
where the potential function (defined up to a constant additive factor) is
\begin{equation}\label{potential}
 V(x)=\frac{\nabla^2 \phi(x)}{2\phi(x)}+E
\end{equation}
namely the potential of a Schr\"odinger equation admitting $\phi$ as
eigenvector with eigenvalue $E$ as it is immediately seen by
rewriting\refeq{potential} as
\begin{equation*}
    -\frac{1}{2}\nabla^2\phi+V\phi=E\phi
\end{equation*}
In this way  the Doob transformation associates a Markov process to every stationary
solution\refeq{statsol} of the Schr\"odinger equation
\begin{equation}\label{schreq}
 i\partial_t \Psi=H\Psi
\end{equation}
with Hamiltonian\refeq{hamiltonian}. When on the other hand we consider a non vanishing $S(x)$ --
namely a \emph{gauge transformation} with respect to the previous case -- then our wave functions
show the complete form\refeq{wavefunct}, and starting again from\refeq{eq:hamiltonianosimmetrico}
a slightly longer calculation shows that the Hamiltonian now is
\begin{equation}
\label{eq:gauge}
    H=\frac{1}{2}(i\nabla+\nabla S)^2+V
\end{equation}
with $V(x)$ always defined as in\refeq{potential}: in other words in this case from $\phi$ and $S$
we get both a scalar and a vector potential. Remark that, despite the presence of the term $\nabla
S$ in\refeq{eq:gauge}, no physical electromagnetic field is actually acting on the particle, as is
well known from the gauge transformation theory. The study of a stochastic description of a
particle subjected to an electromagnetic field is not undertaken here: readers interested in this
argument can usefully refer to~\cite{morato82}.

Similar results can be obtained by initially choosing a constant function $\phi(x)=1$ and
$S(x)=p\cdot x$ so that
\begin{equation*}
    \psi(x)=e^{ip\cdot x}\qquad\quad\Psi(x,t)=e^{ip\cdot x-iEt}
\end{equation*}
namely the wave function of a plane wave. In this case however instead of\refeq{eq:diffoval} we
have to take
\begin{equation*}
    Lf:=L_0f+p\cdot\nabla f=\frac{1}{2}\nabla^2 f+p\cdot\nabla f
\end{equation*}
which is still of the form\refeq{eq:diffovalbis}, albeit with a
constant $b(x)=p$. Since this $L$ is no longer self-adjoint in
$\leb(\re^2,dx)$ because, with an integration by parts, we find
\begin{equation*}
    L^\dag=\tilde{L}=\frac{1}{2}\nabla^2 f-p\cdot\nabla f
\end{equation*}
we now get  from\refeq{k} and\refeq{eq:aggiungi5}
\begin{equation*}
  Kf = -\frac{1}{2}\nabla^2 f-ip\cdot\nabla f \qquad\quad
  K_\psi f = -\frac{1}{2}\nabla^2 f-\frac{p^2}{2} f
\end{equation*}
and then finally -- by choosing, as usual, $E=p^2/2$ -- we obtain
from\refeq{eq:hamiltoniano}
\begin{equation*}
    H=\frac{1}{2}\,\nabla^2
\end{equation*}
which is the Hamiltonian of the Schr\"odinger equation\refeq{schreq} in its free form. Remark that
since $b(x)=p$ the resulting Markov process associated to $\Psi$ now simply is a Wiener process
plus a constant drift, but, at variance with the previous cases, we no longer have normalizable
stationary solutions of\refeq{schreq}, because $\phi^2(x)dx=dx$ defines an invariant measure which
is the Lebesgue measure and not a probability: in other words our Markov process $X_t$ will be an
improper one in the sense outlined in the Section~\ref{se:articolo3}.

 \section{L\'evy-Schr\"odinger Equation}\label{se:articolo5}

In this section we will focus our attention on a form of the generalized Schr\"odinger
equation~(\ref{levyschr}) which, without being the most general one, is less particular than that
discussed in the Section~\ref{sm}: namely the L\'evy-Schr\"odinger equation already introduced in
a few previous papers~\cite{cufaro6} where the Hamiltonian operator was found to be
\begin{equation}
\label{eq:lev-scho}
 H=-L_0+V,\qquad D(H)=D(L_0)\cap D(V)
\end{equation}
with $(L_0,D(L_0))$ infinitesimal generator of a symmetric L\'evy
process taking values in $\re^n$, and $V$ a measurable real function
defined on $\re^n$ that makes the operator $(H,D(H))$ self-adjoint
and bounded from below. We have already seen in (\ref{eq:levval})
and (\ref{eq:levval2}) how the generator $(L_0,D(L_0))$ of a L\évy
process actually operates. We add here that when the L\évy process
is symmetric its logarithmic characteristic $\eta(u)$ is a real
function, and hence from\refeq{eq:levval} we easily have that
$(L_0,D(L_0))$ is self-adjoint in $\leb^2(\re^n,dx)$, while from the
L\'evy-Khintchin formula we also deduce that it is negative
definite~\cite{sato}. In the quoted papers, however, the
choice\refeq{eq:lev-scho} was essentially dictated by an analogy
argument and there was no real attempt to deduce it: here instead we
will try to extend this idea, and to justify it within the framework
of the Doob transformations.

To show the way, we here consider first the case $\phi(x)=1$
associated to the Lebesgue measure $\mu(dx)=dx$. Every L\'evy
process is a time-homogeneous Markov process, and this $\mu$ acts as
its invariant measure~\cite{sato}. It is apparent, moreover, that
here we are not required to introduce a further operator $L$ -- as
we did in\refeq{eq:diffoval} for the Wiener case -- because $L_0$
itself plays this role. As a consequence we can skip to prove a
statement as the Theorem~\ref{th:esempio1}, for $L_0$ is by
hypothesis the generator of a L\évy process. On the other hand, even
if here $\mu$ is only $\sigma$-finite and cannot be considered as a
probability measure, we can apply the Doob transformation defined in
the previous section because all the hypotheses~{1-4} are met. Since
$L_0$ is self-adjoint in $\leb^2(\re^n,dx)$ (and hence for the
backward generator we have $\tilde{L}_0=L_0^\dag=L_0$) from
(\ref{eq:6}) we find
\begin{equation*}
K=-L_0\qquad\qquad K^Y=-L_0-i\partial_t
\end{equation*}
and to implement a Doob transformation -- by taking the Lebesgue
measure as invariant measure, $S(x)=0$ and $E=0$ for simplicity --
we choose
\begin{equation}\label{planew}
\Psi(t,x)=\psi(x)=1
\end{equation}
namely the simplest possible form of a plane wave. As a consequence
in this first case we finally get
\begin{equation*}
H=-L_0
\end{equation*}
namely we find the case $V(x)=0\,$ of\refeq{eq:lev-scho}. Since this Hamiltonian essentially
coincides with our initial generator $L_0$ of a L\évy process, it is straightforward to conclude
-- as in the Proposition~\ref{prop} and the subsequent remarks -- that to a plane
wave\refeq{planew} solution of a free generalized Schr\"odinger equation\refeq{levyschr}, namely
of the free L\évy-Schr\"odinger equation
\begin{equation}\label{freeeq}
    i\partial_t\Psi=H\Psi=-L_0\Psi
\end{equation}
we can simply associate the L\'evy process corresponding to $L_0$. As a matter of fact this will
be an improper process with generator $L_0$ and with the Lebesgue measure $\mu$ as initial -- and
invariant -- measure. The free equation\refeq{freeeq} is the case that has already been discussed
at length -- albeit in a more heuristic framework -- in the previous papers~\cite{cufaro6}. If our
L\évy process is also rotationally invariant then from\refeq{eq:levval3}, and within the notations
of the Section~\ref{se:articolo2}, the Hamilton operator $H$ becomes
\begin{equation}\label{eq:freeham}
 [Hf](x)=-\frac{\alpha}{2}\nabla^2 f(x)-\int_{y\neq0}\left[\delta^2_yf\right](x)\,\ell(dy)
\end{equation}
for any complex Schwarz function $f$ and $x\in\re^{n}$. Note that if the jump term vanishes
(namely if $L_0$ is the generator of a Wiener process) and $\alpha=1$ we have
\begin{equation*}
H=-\frac{1}{2}\,\nabla^2
\end{equation*}
i.e. the free Hamilton operator\refeq{hamiltonian} of the stochastic mechanics presented in the
previous section. As a consequence we see that\refeq{eq:freeham} can be considered as the
generalization of the usual quantum mechanical Hamiltonian by means of a jump term produced by the
possible non Gaussian nature of our background L\évy process.

We turn then our attention to the more interesting case of a non constant $\phi$ leading to a non
vanishing $V$ in\refeq{eq:lev-scho}, and we suppose that $\alpha_{ij}=0$ in\refeq{eq:levval2},
namely that $L_0$ is the generator of a pure jump process without an unessential Gaussian
component (we can always add it later). With this $L_0$ we take now a $\phi\in D(L_{0})$ such that
\begin{equation}\label{positive}
\int_{\re^{n}}\phi^2(x)\,dx=1,\qquad\quad \phi> 0,\ \mbox{a.s. in}\
dx
\end{equation}
and in analogy with\refeq{eq:diffoval} we introduce the new operator
$(L,D(L))$ in $\leb^{2}(\re^{n},d\mu)$
\begin{equation}\label{eq:diffoval2}
  Lf:=\frac{L_{0}(\phi f)-fL_{0}\phi}{\phi}
\end{equation}
with $D(L):=C^{\infty}_{0}(\re^{n})$ and $\mu(dx)=\phi^{2}(x)dx$.
\begin{proposition}\label{le:lemmadom}
If $\phi\in D(L_0)$, then for every $f\in C_0^{\infty}(\re^n)$ we have $\phi f\in D(L_0)$ and
\begin{equation}
\label{eq:levintparti}
 L_0(\phi f)=fL_0\phi+\phi \,L_0 f+\int_{y\neq0}\delta_y \phi\,\delta_y f\,\ell(dy)
\end{equation}
\end{proposition}
\begin{proof} See Appendix~\ref{appendix}.
\end{proof}\\
This statement -- which generalizes an \emph{integration by parts} rule -- proves first that our
definition\refeq{eq:diffoval2} is consistent in the sense that $\phi f\in D(L_0)$; then
from\refeq{eq:diffoval2} and\refeq{eq:levintparti} it also gives the jump version
of\refeq{eq:diffovalbis}
\begin{equation*}
    Lf=L_0f+\int_{y\neq0}\frac{\delta_y \phi}{\phi}\,\delta_y f\,\ell(dy)
\end{equation*}
so that, by taking into account the equation\refeq{eq:levval3} with
$a=0$, we find
\begin{eqnarray}\label{generator}
    [Lf](x)&=&\int_{y\neq0}\left(\delta_y^2f+\frac{\delta_y \phi}{\phi}\,\delta_y
    f\right)\ell(dy)\nonumber\\
    &=&\int_{y\neq0}\left(\delta_yf-\frac{\phi\bm1_{B_1}(y)}{\phi+\delta_y\phi}\,y\cdot\nabla f\right)\frac{\phi+\delta_y\phi}{\phi}\,\ell(dy)\nonumber\\
    &=&\int_{y\neq0}\left[f(x+y)-f(x)-\gamma(x,y)\,y\cdot\nabla f(x)\right]\lambda(x;dy)
\end{eqnarray}
where
\begin{equation}\label{coeff}
    \gamma(x,y)=\frac{\phi(x)}{\phi(x+y)}\,\bm1_{B_1}(y)
    \qquad\qquad\lambda(x;dy)=\frac{\phi(x+y)}{\phi(x)}\,\ell(dy)
\end{equation}
The equation\refeq{generator} explicitly shows that the generator $L$ introduced
in\refeq{eq:diffoval2} is a \emph{L\évy-type} operator~\cite{applebaum,jacob}. We then state our
main result on the existence of a (L\évy-type) Markov process associated to $L$ and, through a
subsequent Doob transformation, to the (stationary) solutions of a L\évy--Schr\"odinger equation.
\begin{theorem}\label{th:esempio2}
The operator $(L,D(L))$ defined in\refeq{eq:diffoval2} is closable and its closure
$(\bar{L},D(\bar{L}))$ is a self-adjoint, negative definite operator which is the infinitesimal
generator of a Markov process $(X_t)_{t\in\re}$. The measure $\mu(dx)$ is invariant for this
Markov process.
\end{theorem}
\begin{proof}
As in the Theorem \ref{th:esempio1} from $(L,D(L))$ we first define
in $\leb^2(\re^n,d\mu)$ the bilinear form
\begin{equation*}
\mathcal{E}(f,g):=-\langle Lf,g\rangle_\mu
\end{equation*}
with $D(\mathcal{E}):=C^{\infty}_{0}(\re^{n})$ and we remark that from\refeq{eq:diffoval2} we have
\begin{equation*}
\dir(f,g)=-\int \frac{L_{0}(\phi f)-fL_{0}\phi}{\phi}\,g\phi^2\,dx=
-\int\phi\, g L_0(f\phi)\, dx+ \int \phi\, gfL_0\phi\,dx
\end{equation*}
Then from\refeq{eq:levdirval} with $\alpha_{ij}=0$ we  obtain
\begin{eqnarray*}
  \dir(f,g)&=&  \frac{1}{2}\!\int\!\!\int_{y\neq 0}[\delta_y(g\phi)](x)\,[\delta_y(f\phi)](x)\,\ell(dy)dx \\
           &&\qquad\qquad\qquad\quad  -\frac{1}{2}\!\int\!\!\int_{y\neq0} [\delta_y(fg\phi)](x)\,[\delta_y\phi](x)\,\ell(dy)dx
\end{eqnarray*}
and from\refeq{eq:propr_delta} after some tiring but simple algebra
we get
\begin{equation}
\label{eq:dirifine}
\dir(f,g)=\frac{1}{2}\int\!\!\int_{y\neq0}[\delta_yf](x)[\delta_yg](x)\phi(x+y)\phi(x)\,\ell(dy)dx
\end{equation}
This expression for $\dir$ has the required form\refeq{eq:simmform}
with vanishing killing and diffusion measures, and $J(dz,dx)$ given
by $\phi(x+y)\phi(x)\ell(dy)dx$ with $z=x+y$, which is positive
because of\refeq{positive}. Since the
condition\refeq{eq:verificadiri1} is satisfied, by reproducing the
same argument previously adopted in the Theorem~\ref{th:esempio1} we
get that there exists a self-adjoint, infinitesimal generator
$(\bar{L},D(\bar{L}))$ which turns out~\cite{reed} to be the closure
of $(L,D(L))$.

This $\bar{L}$ generates a unique, bilateral Markov process $(X_{t})_{t\in\re}$ on
$(\re^{n},\mathcal{B}(\re^{n}))$ having $\mu$ as its invariant measure when the conditions
discussed in the remarks following the Theorem~\ref{th:dirimarkov} are met. Hence, as in the
Theorem~\ref{th:esempio1}, we should only check that the constant element $f_1(x)=1$ of
$\leb^2(\re^n,d\mu)$ belongs to $D(\bar{L})$ and that $\bar{L}f_1=0$. In fact for every $f\in
C^{\infty}_{0}(\re^{n})$ dense in $\leb^2(\re^n,d\mu)$  we have again from\refeq{eq:levdirval}
with $\alpha_{ij}=0$ that
\begin{eqnarray*}
\langle f_1,\bar{L}f\rangle_\mu&=&\int_{\re^n} [\bar L f ]\phi^2dx=\int_{\re^n} [L f] \phi^2dx\\
 &=&\int_{\re^n}\phi^2 \frac{L_0(\phi f)-fL_0(\phi)}{\phi}\, dx=\int_{\re^n}\left[\phi L_0(\phi f)-\phi fL_0(\phi)\right]\,dx\\
 &=&-\frac{1}{2}\int\!\!\int_{y\neq0}\left[\delta_y\phi\,\delta_y(\phi f)-\delta_y(\phi f)\delta_y\,\phi\right ]\,dx=0
\end{eqnarray*}
Being $\bar L$ self-adjoint this implies first that $f_1\in D(\bar
L)$, and that $\langle \bar{L}f_1,f\rangle_\mu=0$ for every $f\in
C_0^{\infty}$ dense in $\leb^{2}(\re^{n},d\mu)$, and then that $\bar
Lf_1=0$.
\end{proof}


This result will now put us in condition to perform a suitable Doob
transformation with the confidence that we can also associate a
L\évy-type, Markov process $X$ to the chosen wave functions. In
fact, being $\mu(dx)=\rho(x)dx=\phi^2(x)dx$ an invariant measure for
$X_t$, taking as before $S(x)=0$ and $E\in\re$ namely
\begin{equation*}
 \psi(x)=\phi(x)\qquad\qquad\Psi(t,x)=e^{-iEt}\phi(x)
\end{equation*}
and by applying (\ref{eq:hamiltonianosimmetrico}), we immediately
get the following expression for the Hamiltonian operator associated
to $X_{t}$ of Theorem~\ref{th:esempio2}
\begin{equation}\label{LSham}
 [Hf](x):=[-L_0f+Vf](x)=-\int_{y\neq0}[\delta_y^2f](x)\,\ell(dy)+V(x)f(x)
\end{equation}
where the potential function now is
\begin{equation}\label{potential2}
 V(x)=\frac{[L_0\phi](x)}{\phi(x)}+E=\frac{1}{\phi(x)}\int_{y\neq0}[\delta_y^2\phi](x)\,\ell(dy)+E
\end{equation}
showing also that $\Psi$ is a stationary solution of
\begin{equation}\label{LSeqV}
 i\partial_t\Psi=H\Psi=(-L_0+V)\Psi=-\int_{y\neq0}\delta^2_y\Psi\,\ell(dy)+V\Psi
\end{equation}
namely of our L\évy-Schr\"odinger equation with the
potential\refeq{potential2}.

\section{Cauchy noise}\label{cauchy}

We will conclude the paper by proposing (in the one-dimensional case
$n=1$) two examples for the simplest stable, non Gaussian L\évy
background noise produced by the generator $L_0$ of a Cauchy
process, namely an operator of the forn\refeq{eq:levval3} without
Gaussian term ($a=0$) and with L\évy measure
\begin{equation}\label{cauchyLM}
    \ell(dx)=\frac{dx}{\pi x^2}
\end{equation}
Remark that in general for this Cauchy background noise with L\évy measure\refeq{cauchyLM} the
generator\refeq{generator} becomes
\begin{equation}\label{generator1}
    Lf=\frac{1}{\pi}\int_{y\neq0}\left(\frac{\delta_y^2f}{y^2}+\frac{1}{y}\frac{\delta_y
\phi}{\phi}\,\frac{\delta_y f}{y}\right)dy
\end{equation}
and that the convergence of this integral in $y=0$ is a consequence of the fact that, for $y\to0$,
$\delta_y^2f$ vanishes at the second order, while $\delta_y \phi$ and $\delta_y f$ are
infinitesimal of the first order. The corresponding pure jump Cauchy-Schr\"odinger equation
\begin{equation}\label{causchr}
    i\partial_t\Psi=-\int_{y\neq0}\frac{\delta^2_y\Psi}{\pi
    y^2}\,dy+V\Psi
\end{equation}
has already been discussed in various disguises in several previous papers~\cite{cufaro6,garbacz}
and we will show here two examples of its stationary solutions for potentials of the
form\refeq{potential}.

\begin{figure}
\begin{center}
\includegraphics*[width=10cm]{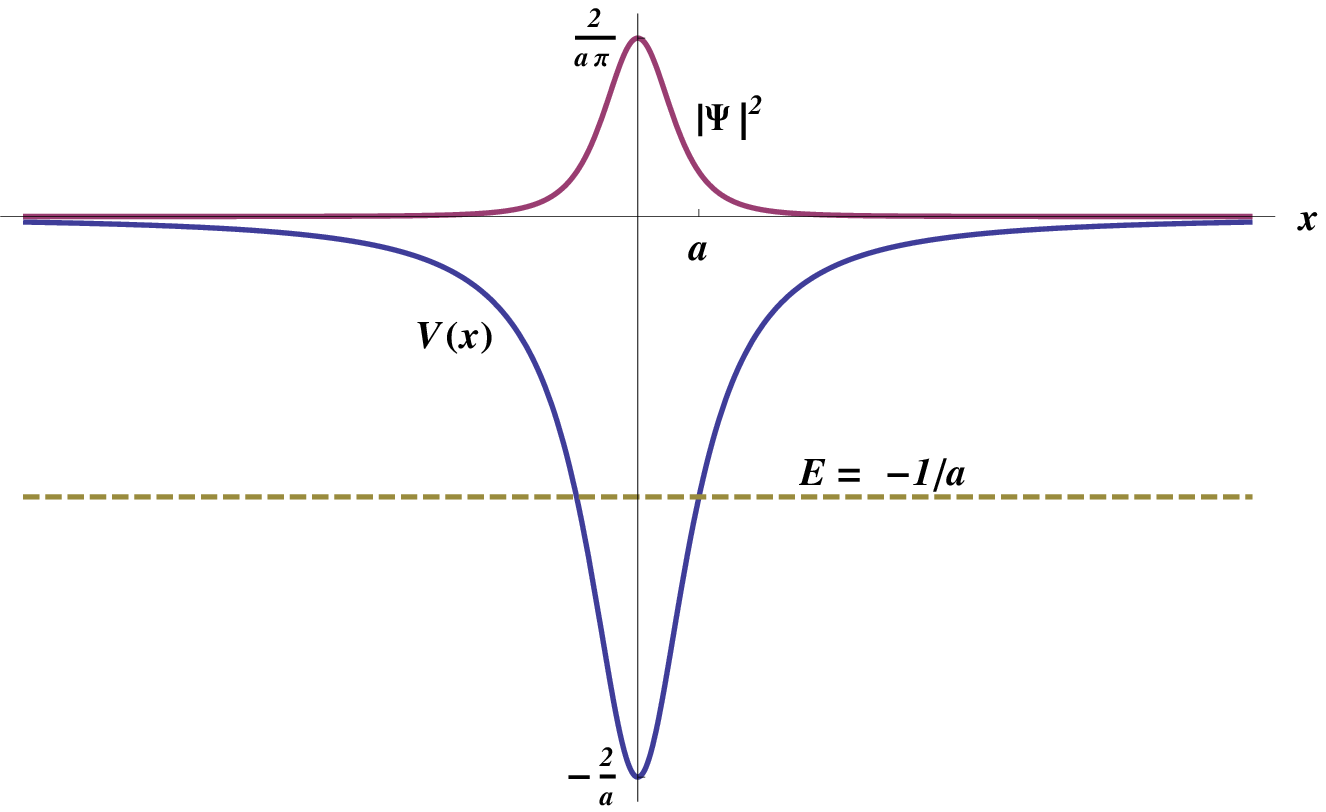}
\caption{Potentials $V(x)$ and square modulus of the stationary wave functions $|\Psi|^2$ for the
pair\refeq{pot1} and\refeq{wf1}.}\label{fig01}
\end{center}
\end{figure}
To define our invariant measure $\mu(dx)=\rho(x)dx=\phi^2(x)dx$ let us take first of all the
functions
\begin{equation}\label{ampl1}
    \phi(x)=\sqrt{\frac{2a}{\pi}}\frac{a}{a^2+x^2}\qquad\qquad\rho(x)=\phi^2(x)=\frac{2}{a\pi}\left(\frac{a^2}{a^2+x^x}\right)^2
\end{equation}
so that the stationary \emph{pdf} will be that of a $\mathfrak{T}_a(3)$ Student
law~\cite{cufaro6}. A direct calculation of\refeq{potential} with these entries will then show
that, by choosing the energy origin so that $E=-1/a$ and $V(\pm\infty)=0$, we have
\begin{equation}\label{pot1}
    V(x)=-\,\frac{2a}{x^2+a^2}
\end{equation}
In other words the wave function
\begin{equation}\label{wf1}
    \Psi(x,t)=\sqrt{\frac{2a}{\pi}}\frac{a}{a^2+x^2}\,e^{-it/a}
\end{equation}
turns out to be a stationary solution of the equation\refeq{causchr} with potential\refeq{pot1}
corresponding to the eigenvalue $E=-1/a$. This result is summarized in the Figure~\ref{fig01}. On
the other hand we see from the Theorem~\ref{th:esempio2} that to the wave function\refeq{wf1} we
can also associate a L\évy-type, Markov process $X_t$ completely defined by yhe
generator\refeq{generator} with
\begin{equation*}
    \gamma(x,y)=\frac{a^2+(x+y)^2}{a^2+x^2}\,\bm 1_{[-1,1]}(y)\qquad\qquad\lambda(x;dy)=\frac{a^2+x^2}{a^2+(x+y)^2}\,\frac{dy}{\pi y^2}
\end{equation*}
as can be deduced from\refeq{coeff} and\refeq{ampl1}.

\begin{figure}
\begin{center}
\includegraphics*[width=10cm]{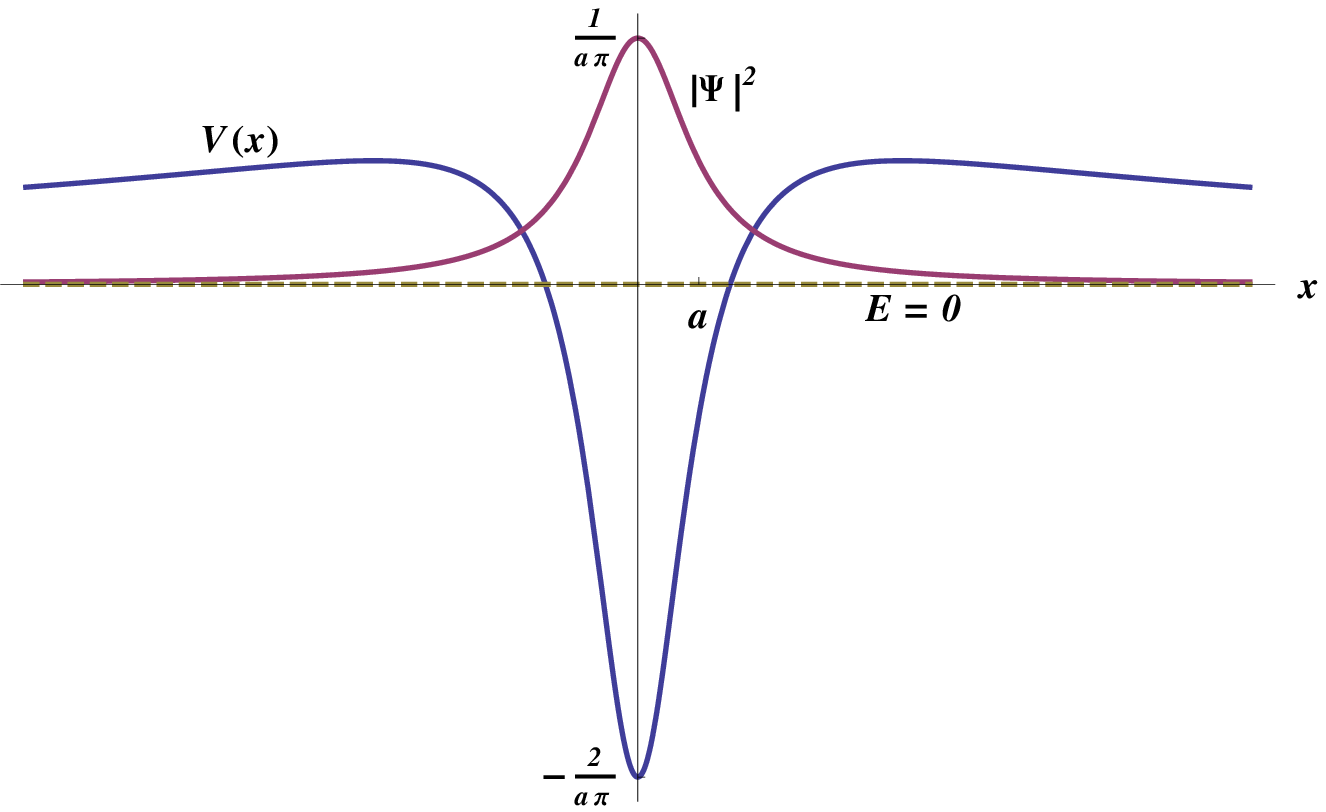}
\caption{Potentials $V(x)$ and square modulus of the stationary wave functions $|\Psi|^2$ for the
pair\refeq{pot2} and\refeq{wf2}.}\label{fig02}
\end{center}
\end{figure}
In a similar vein, and always for the same equation\refeq{causchr}, we can take as a second
example the starting functions
\begin{equation*}
   \phi(x)=\sqrt{\frac{a}{\pi}}\frac{1}{\sqrt{a^2+x^2}}\qquad\qquad\rho(x)=\phi^2(x)=\frac{1}{\pi}\frac{a}{a^2+x^2}
\end{equation*}
namely the \emph{pdf} of a $\mathfrak{C}(a)=\mathfrak{T}_a(1)$ Cauchy law. Treading the same path
as before, a slightly more laborious calculation will  show that, by choosing now $E=0$ to have
again $V(\pm\infty)=0$, we find
\begin{equation}\label{pot2}
    V(x)=-\,\frac{2}{\pi}\left[\frac{1}{\sqrt{a^2+x^2}}+\frac{x}{a^2+x^2}\log\left(\sqrt{1+\frac{x^2}{a^2}}-\frac{x}{a}\right)\right]
\end{equation}
so that now the wave function
\begin{equation}\label{wf2}
    \Psi(x,t)=\sqrt{\frac{a}{\pi}}\frac{1}{\sqrt{a^2+x^2}}
\end{equation}
is stationary solution with eigenvalue $E=0$ of the equation\refeq{causchr} with
potential\refeq{pot2}. The potential and the corresponding \emph{pdf} are depicted in the
Figure~\ref{fig02}. The L\évy-type, Markov process $X_t$ associated to this wave function is again
defined by the generator\refeq{generator} with
\begin{equation*}
    \gamma(x,y)=\sqrt{\frac{a^2+(x+y)^2}{a^2+x^2}}\,\bm 1_{[-1,1]}(y)\qquad\qquad\lambda(x;dy)=\sqrt{\frac{a^2+x^2}{a^2+(x+y)^2}}\,\frac{dy}{\pi y^2}
\end{equation*}
The two generators  introduced here completely determine the two Markov processes associated to
our stationary solutions of the Cauchy-Schr\"odinger equation\refeq{causchr}.

\section{Conclusions}\label{se:articolo6}

The adoption, proposed in a few previous papers~\cite{cufaro6}, of
the \LS\ equation -- a generalization of the usual Schr\"odinger
equation associated to the Wiener process -- amounts in fact to
suppose that the behavior of the physical systems in consideration
is based on an underlying L\'evy process that can have both Gaussian
(continuous) and non Gaussian (jumping) components. The consequent
use of all the gamut of the \id, even non stable, processes on the
other hand turns out to be important and physically meaningful
because there are significant cases that fall in the domain of the
\LS\ picture, without being in that of a stable (fractional)
Schr\"odinger equation~\cite{laskin}. In particular the simplest
form of a relativistic, free Schr\"odinger
equation~\cite{cufaro6,garbacz,ichinose} can be associated with a
peculiar type of self-decomposable, non stable process acting as
background noise. Moreover in many instances of the \LS\ equation
the resulting energy-momentum relations can be seen as small
corrections to the classical relations for small values of certain
parameters~\cite{cufaro6}. It must also be remembered that -- in
discordance with the stable, fractional case -- our models are not
tied to the use of background noises with infinite variances: these
can indeed be finite even for purely non Gaussian noises -- as for
instance in the case of the relativistic, free Schr\"odinger
equation -- and can then be used as a legitimate measure of the
dispersion. Finally let us recall that a typical non stable, Student
L\'evy noise seems to be suitable for applications, as for instance
in the models of halo formation in intense beam of charged particles
in accelerators~\cite{applications,cufaro2,vivoli}.

In view of all that it was then important to explicitly give more
rigorous details about the formal association between \LS\ wave
functions and the underlying L\évy processes, namely a true
generalized stochastic mechanics. And it was urgent also to explore
this L\évy--Nelson stochastic mechanics by adding suitable
potentials to the free \LS\ equation, and by studying the
corresponding possible stationary and coherent states. To this end
we found expedient to broaden the scope of our enquiry to the field
of Markov processes more general than the L\évy processes

From this standpoint in the present paper we have studied -- with the aid of the Dirichlet forms
-- under what conditions the generalized Schr\"o\-din\-ger equation\refeq{levyschr}, with a fairly
general self-adjoint Hamiltonian $H$, admits a stochastic representation in terms of Markov
processes: a conspicuous extension of the well known, older results of the stochastic
mechanics~\cite{nelson}. More precisely it has been shown how we can associate to every stationary
wave functions $\Psi$, of the form\refeq{wavefunct} and solutions of the equation\refeq{levyschr},
a bilateral, time-homogeneous Markov process $(X_t)_{t\in\re}$ whose generator $L$ in its turn
plays the role of the starting point to produce exactly the Hamiltonian $H$ of the
equation\refeq{levyschr}. This association moreover is defined in such a way that, in analogy with
the well known Born postulate of quantum mechanics, $|\Psi|^2$ always coincides with the
\emph{pdf} of $(X_t)_{t\in\re}$.

The whole procedure adopted here is inherently based on the Doob
transformation\refeq{eq:tipo-doob} previously
adopted~\cite{albeverio,morato} in the particular case of the Wiener
process to get the usual Schr\"odinger equation, as we have
summarized in the Section~\ref{sm}. This choice allows us, among
other things, to sidestep for the time being the problem of the
explicit definition of a dynamics for jump processes that would pave
the way to recover our association along a more traditional path,
either by means of the Newton law, as originally done~\cite{nelson},
or through a variational approach, as in later
advances~\cite{guerra}. The definition of these structures, whose
preliminary results have already been presented
elsewhere~\cite{andrisani}, seems indeed at present to require more
ironing and will be the object of future enquiries.

We have then focused our attention on the case of the \LS\ equation\refeq{LSeqV}, a particular
kind of generalized Schr\"odinger equation characterized by an Hamiltonian\refeq{LSham} derived
from a L\évy process generator. This equation was previously suggested only in an heuristic
way~\cite{cufaro6}, while in the present paper we succeeded in proving a few rigorous results:
first we showed how the stationary wave functions of this \LS\ equation actually satisfy the
conditions required to be associated to Markov processes; then we pointed out that the \LS\
Hamiltonian $H$ is composed of a kinetic part (the generator $L_0$ of our background, symmetric
L\évy process) plus a potential $V$, a term that was lacking of a justification in our previous
formulations. Third we also proved that the bilateral, time-homogeneous Markov processes
associated to a stationary wave function $\Psi$ turn out in fact to be L\évy-type processes, a
generalization of the L\évy processes which is at present under intense scrutiny~\cite{jacob}.
Finally we presented a few examples of stationary solutions of \LS\ equations with Cauchy
background noise.

These were much needed advances conspicuously absent in the previous
papers, as already explicitly remarked there, It would be important
now first to extend these results even to the non stationary wave
packets solutions of the generalized Schr\"odinger
equation\refeq{levyschr}, at least in its \LS\ form\refeq{LSeqV},
that have been extensively studied in a recent paper~\cite{cufaro6}
where their inherent muli-modality has been put in evidence. Then to
give a satisfactory formulation of the Nelson dynamics of the jump
processes involved: a much needed advance that would constitute an
open window on the true nature of these special processes. And
finally a detailed study of the characteristics of the L\évy-type
processes associated to the \LS\ wave functions: this too will be
the subject of future papers.

\begin{appendix}

\section{Proof of Proposition~\ref{le:lemmadom}}\label{appendix}

We begin by proving\refeq{eq:levintparti} for $\phi\in
C_0^{\infty}(\re^n)\subseteq\mathcal{S}(\re^n)\subseteq D(L_0)$. In fact in this case we
apparently have $\phi f\in C_0^{\infty}(\re^n)$, while from\refeq{eq:levval2}
and\refeq{eq:provalemma0} with $A=0$ it is
\begin{eqnarray}\label{eq:provalemma}
 L_0(\phi f)&=&\int_{y\neq0} \delta_y^2(\phi f)\,\ell(dy)\nonumber\\
            &=&f\int_{y\neq0} \delta_y^2\phi\,\ell(dy)+\phi\int_{y\neq0} \delta_y^2f\,\ell(dy)+\int_{y\neq0} \delta_y\phi\,\delta_y
            f\,\ell(dy)\nonumber\\
            &=&fL_0\phi+\phi L_0f+\int_{y\neq0} \delta_y\phi\,\delta_y f\,\ell(dy)
\end{eqnarray}
because we can see that the third term of\refeq{eq:provalemma} belongs to
$\leb^2_{\re}(\re^n,dx)$. Being indeed $f$ and $\phi$ bounded, and $\phi\in
C_0^{\infty}(\re^n)\subseteq D(L_0)\subseteq D(\dir_0)$, from\refeq{eq:levdirdom} we get
\begin{eqnarray*}
 \int\left|\int_{y\neq0} \delta_yf\delta_y\phi\,\ell(dy)\right|^2dx&\leq&  \int\!\!\int_{y\neq0}
 \left|\delta_yf\,\delta_y\phi\right|^2\ell(dy)dx\\
 &\leq& 4\kappa^2\int\!\!\int_{y\neq0} |\delta_y\phi|^2\,\ell(dy)dx<\infty
\end{eqnarray*}
with $\kappa=\sup |f|$.

If instead we suppose that $\phi\in D(L_0)\subseteq D(\dir_0)$, we have $\phi
f\in\leb^2_{\re}(\re^n,dx)$ since $f$ is bounded. As a first step let us show that $\phi f\in
D(\dir_0)\supseteq D(L_0)$: because of\refeq{eq:levdirdom} to do that we have just to prove that
\begin{equation}
 \label{eq:provalemma1}
\int\!\!\int_{y\neq0}|\delta_y(\phi f)|^2\,\ell(dy)dx<\infty
\end{equation}
Remark that from\refeq{eq:propr_delta} we have
\begin{eqnarray}
\label{eq:provalemma2}
 \int\!\!\int_{y\neq0}|\delta_y(\phi f)|^2\,\ell(dy)dx&\leq& 2\left[\int\!\!\int_{y\neq0}|\phi\,\delta_y f|^2\,\ell(dy)dx+ \int\!\!\int_{y\neq0}|f\delta_y \phi|^2\,\ell(dy)dx\right.\nonumber\\
&&\qquad\quad\left.+\int\!\!\int_{y\neq0}|\delta_y\phi\delta_y f|^2\,\ell(dy)dx\right]
\end{eqnarray}
The second and the third integrals on the right of (\ref{eq:provalemma2}) are finite because $f$
is bounded and $\phi\in D(\dir_0)$; as for the first integral, instead, since
$\phi\in\leb^2_{\re}(\re^n,dx)$, it is enough to remark that, from the typical property of L\évy
measures
\begin{equation*}
    \int_{y\neq0}\left(|y|^2\wedge1\right)\,\ell(dy)<\infty
\end{equation*}
 and since $f\in C_0^\infty(\re^n)$, we have
\begin{eqnarray*}
 \int_{y\neq0}|\delta_yf|^2(x)\,\ell(dy)&=&\int_{|y|\leq1,\,y\neq0}|\delta_yf|^2(x)\,\ell(dy)+\int_{|y|\geq1}|\delta_yf|^2(x)\,\ell(dy)\\
 &\leq& C\int_{|y|\leq1,\,y\neq0}|y|^2|\nabla f(x)|^2\,\ell(dy)+C'\int_{|y|\geq1}\ell(dy)\\
 &\leq& C''\int_{|y|\leq1,\,y\neq0}|y|^2\,\ell(dy)+C'\int_{|y|\geq1}\ell(dy)=M<\infty
\end{eqnarray*}
with $C''=C\sup|\nabla f(x)|^2$. In a similar way we can prove that also $L_0f$ is a bounded
function, a remark that will be useful in the following.

In order to complete the proof of the proposition we will show now that it exists a $\psi\in
\leb_{\re}^2(\re^n,dx)$ with the form of the second member of\refeq{eq:levintparti}, and such that
for every $g\in D(L_0)$
\begin{equation*}
 \langle L_0 g,\phi f\rangle=\langle g,\psi\rangle
\end{equation*}
If this is true we can indeed deduce form the self-adjointness of $L_0$ that $f\phi\in D(L_0)$,
and that\refeq{eq:levintparti} is verified. On the other hand, since the Proposition
\ref{pr:closable-diri} actually states that $C_0^{\infty}(\re^n)$ is a \emph{core} for
$(L_0,D(L_0))$, we can restrict our discussion to $g\in C_0^{\infty}(\re^n)$, so that we have
\begin{eqnarray}\label{eq:provalemma3}
 \langle L_0 g,\phi f\rangle&=&\langle f L_0 g,\phi \rangle\nonumber\\
      &=&\langle  L_0(fg),\phi \rangle-\langle g L_0 f,\phi \rangle-\int\!\! \phi\left(\int_{y\neq0}\delta_y g\,\delta_y f\,\ell(dy)\right) dx\nonumber\\
      &=&\langle g,fL_0\phi\rangle-\langle g,\phi L_0 f\rangle-\int\!\!\int_{y\neq0}\phi\,\delta_y
g\,\delta_y f\,\ell(dy) dx
\end{eqnarray}
where the second equality follows from the fact that (\ref{eq:levintparti}) has been already
proved for $f,g\in C_0^{\infty}(\re^n)$, while the third equality comes from the previously quoted
$L_0f$ boundedness (so that $\phi L_0f\in\leb_{\re}^2(\re^n,dx)$ and the scalar products exist)
and the $L_0$ self-adjointness. Take now the third term of\refeq{eq:provalemma3}: according
to\refeq{eq:propr_delta} we have
\begin{equation}
\label{eq:provalemma3bis}
 \phi\,\delta_y g\,\delta_y f=\delta_y (\phi g)\,\delta_y f-g\,\delta_y \phi\,\delta_y f-\delta_y\, \phi\,\delta_y f\,\delta_y g
\end{equation}
Since we have seen that $\phi g\in D(\dir_0)$, from\refeq{eq:levdirval} we first find
\begin{equation} \label{eq:provalemma4}
\int\!\! \int_{y\neq0}\delta_y (\phi g)\,\delta_y f\,\ell(dy) dx=-2\dir_0(\phi g,f)=-2\langle \phi
g,L_0f\rangle=-2\langle g,\phi L_0f\rangle
\end{equation}
Then, taking now into account that
\begin{equation*}
 \int\!\! \int_{y\neq0}\big|\delta_y \phi \,\delta_y f\,\delta_y g\big|^2\,\ell(dy) dx\leq C \int \!\!\int_{y\neq0}|\delta_y \phi|^2\,\ell(dy) dx <\infty
\end{equation*}
from the Fubini's theorem and the symmetry of $\ell(dy)$ it results with the change of variables
$x\rightarrow x+y$ and $y\rightarrow -y$ that
\begin{equation}
\label{eq:provalemma5}
 \int\!\! \int_{y\neq0}\delta_y \phi \,\delta_y f\,\delta_y g\,\ell(dy) dx=-\int\!\! \int_{y\neq0}\delta_y \phi \,\delta_y f\,\delta_y g\,\ell(dy) dx=0
\end{equation}
Finally, by observing that $f$ is bounded and $\phi\in D(\dir_0)$, we have
\begin{equation*}
  \int_{y\neq0}[\delta_y \phi ](x)[\delta_y f](x)\,\ell(dy)\;\in\;\leb_{\re}^{2}(\re^n,dx)
\end{equation*}
So, by collecting\refeq{eq:provalemma3},\refeq{eq:provalemma3bis},\refeq{eq:provalemma4}
and\refeq{eq:provalemma5}, we get
\begin{equation*}
 \langle L_0 g,\phi f\rangle=\left\langle g\,,\,fL_0\phi +\phi L_0f+\!\!\int_{y\neq0}\!\!\delta_y \phi \,\delta_y f\,\ell(dy)\right\rangle
\end{equation*}
which completes the proof.

\end{appendix}
\vspace{10pt}
 \noindent\textsc{Acknowledgements:} The authors want
to thank L.M.\ Morato for invaluable suggestions and discussions.

\end{document}